\documentclass[12pt, letterpaper]{amsart}
\usepackage[left=1in,right=1in,bottom=1in,top=1in]{geometry}
\usepackage{graphicx}
\usepackage{amsmath,amsthm,amssymb}
\usepackage{mathrsfs, mathtools}
\usepackage[all,cmtip]{xy}
\usepackage{cite}
\usepackage{dsfont}
\usepackage{setspace}
\usepackage{youngtab}
\usepackage{amsmath,amsthm,amssymb,booktabs,mathtools}
\usepackage{dsfont}
\usepackage{enumitem}
\usepackage{hyperref}
\usepackage{graphicx}
\usepackage{caption}

\newtheorem{theorem}{Theorem}
\newtheorem{thm}{Theorem}[section]
\newtheorem{lemma}[thm]{Lemma}
\newtheorem{definition}[thm]{Definition}
\newtheorem{corollary}[thm]{Corollary}
\newtheorem{claim}[thm]{Claim}

\onehalfspace

\def\.{\hskip.06cm}

\newcommand{\ccirc}{\mathbin{\mathchoice
  {\xcirc\scriptstyle}
  {\xcirc\scriptstyle}
  {\xcirc\scriptscriptstyle}
  {\xcirc\scriptscriptstyle}
}}
\newcommand{\xcirc}[1]{\vcenter{\hbox{$#1\circ$}}}
\newcommand{\E}{\mathbb{E}}
\newcommand{\p}{\mathbb{P}}

\newcommand{\m}{\mu_{n, q}}
\newcommand{\n}{\mu_{n, q_n}}
\newcommand{\inv}{l}
\newcommand{\IV}{\text{Inv}}
\newcommand{\var}{\text{Var}}

\newcommand{\lis}{\text{LIS}}

\newcommand{\lcs}{\text{LCS}}
\newcommand{\bs}{\boldsymbol}

\begin{document}

\title[Limit Theorems for the Length of the LCS of Mallows Permutations]{Limit Theorems for the Length of the Longest Common Subsequence of Mallows Permutations}
\author{Naya Banerjee$^\dagger$ \and Ke Jin$^\ddagger$}

\thanks{\thinspace ${\hspace{-.45ex}}^\dagger$Mathematical Sciences,
University of Delaware, Newark, DE, 19716.
\hskip.06cm
Email:
\hskip.06cm
\texttt{naya@math.udel.edu}}

\thanks{\thinspace ${\hspace{-.45ex}}^\ddagger$Mathematical Sciences,
University of Delaware, Newark, DE, 19716.
\hskip.06cm
Email:
\hskip.06cm
\texttt{kejin@math.udel.edu}}

\maketitle

\begin{abstract}
The Mallows measure is measure on permutations which was introduced by Mallows in connection with ranking problems in statistics. Under this measure, the probability of a permutation $\pi$ is proportional to $q^{Inv(\pi)}$ where $q$ is a positive parameter and $Inv(\pi)$ is the number of inversions in $\pi$. We consider the length of the longest common subsequence (LCS) of two independently permutations drawn according to $\mu_{n,q}$ and $\mu_{n,q'}$ for some $q,q' >0$.

We show that when $0<q,q'<1$, the limiting law of the LCS is Gaussian. In the regime that $n(1-q) \to \infty$ and $n(1-q') \to \infty$ we show a weak law of large numbers for the LCS. These results extend the results of \cite{Basu} and \cite{Naya} showing weak laws and a limiting law for the distribution of the longest increasing subsequence to showing corresponding results for the longest common subsequence.
\end{abstract}

\section{Introduction}

The length of the longest common subsequence (LCS) of two strings is a measure of their similarity. It is related to the edit distance, which quantifies the number of operations such as insertion, deletion or substitution that are required to transform one string to the other. Calculating the similarity between sequences is a problem that arises naturally in applications such as natural language processing, linguistics, and DNA and protein alignment \cite{Cap12, Pev00, KruSan83, Wat95}.

The LCS has been studied intensively from an algorithmic perspective in computer science and bioinformatics, but there are fewer theoretical results on the asymptotic behavior and laws of the LCS for random sequences. One of the first results is due to Chv\'{a}tal and Sankoff \cite{Sankoff} who showed that the expected length of the LCS of two random $k$-ary sequences of length $n$ when normalized by $n$ converges to a constant $\gamma_k$. Several authors have attemped to determine $\gamma_k$ \cite{Deken, Dancik, Paterson, Lueker} but only bounds are known and the precise value of the limiting constant remains unknown for all $k$.

In this work we focus on the LCS of two random permutations. This problem can be seen to be related to the problem of finding the longest increasing subsequence (LIS) as follows. By relabeling, the LCS of two independent uniformly random permutations has the same distribution as the LIS of a permutation drawn from the uniform measure. However, this argument no longer holds if the permutations are not drawn from the uniform distribution. Recently, the Mallows distribution on permutations has been the subject of much study in the context of monotone subsequences in permutations. The Mallows distribution weighs a permutation exponentially in a real parameter $q>0$ by the number of inversions in the permutation. Asymptotically, the LIS varies as $q$ varies as a function of $n$. When $n(1-q) \to \beta$ for a constant $\beta$, Mueller and Starr \cite{MuellerStarr} showed that, as in the uniform case when $q=1$, $LIS(\pi)/\sqrt n$ tends to a limiting constant for which they give an explicit formula as a function of $\beta$. On the other hand, Bhatnagar and Peled \cite{Naya} have shown that in the regime where $n(1-q) \to \infty$, the LIS scales as $n \sqrt{1-q}$, at the level of a weak law of large numbers. Mallows permutations have a regenerative structure and this has been exploited to show a central limit theorem for the LIS in the case when $q$ is constant. Recently, Pitman and Tang \cite{PitTan17} have extended some of the results on regeneration times for Mallows permutations to other families of distributions with regenerative structure.

To our knowledge, one of the first works studying the LCS for independent permutations drawn from a non-uniform distribution was by Jin \cite{KeEm,Ke2}. In \cite{KeEm} Jin defined a collection of points corresponding to the two permutations and showed that when permutations are drawn from the Mallows measure with parameters $q,q'$, in the regime that $n(1-q
) \to \beta$ and $n(1-q') \to \gamma$, the empirical distribution of the points converges to a density that can be written in terms of the density of Mallows distributed points which Starr derived in \cite{Starr}. Moreover, the LCS of the random permutations is given by the length of the LIS of this collection of points. Using this, \cite{Ke2} proved a weak law of large numbers for the LCS in the regime that $n(1-q
) \to \beta$ and $n(1-q') \to \gamma$. In this regime, the proof is based on obtaining estimates of the numbers of points in small boxes, along the lines of Deuschel and Zeitouni's \cite{DZ95} results on the LIS of iid point processes.

In this work, we obtain results on the aymptotics in the regime that at least one of the parameters, say $q$, is such that $n(1-q
) \to \infty$. In our first main result, we show a weak law for the LCS in this regime. We build on the work of \cite{Naya} where a weak law was shown for the LIS of a random Mallows permutation where $n(1-q
) \to \infty$. In that work, the weak law for the LIS followed from the observation that in thin strips, the points are distributed effectively as Mallows permutations with a parameter $q'$ such that $n(1-q') \to \beta$. Inside the strip, the weak law shown by Mueller and Starr \cite{MuellerStarr} can be applied to the points and since $q$ is small enough in this regime, the LIS can be shown to be approximated by the sum of the LISs in the strips. A similar strategy can be applied to the points in the box whose LIS gives the LCS. The main technical contribution here is the construction of a coupling that allows us to bound the LCS of two independent Mallows distributed permutations by the LIS of a Mallows distributed permutation and a combinatorial result (Lemma \ref{LE67}) that allows us to extend the inequality to the restriction of coupled permutations to a carefully chosen subsequence.

In our second main result, we show that when $q,q'$ are constant, the LCS when appropriately scaled converges to the Gaussian distribution. In this case, we build on the approach used in \cite{Basu} to show a Gaussian limit theorem for the LIS in the regime that $0< q<1$. The main contribution here is to provide estimates on the return times of a product chain which gives the length of the LCS in analogy to the how such estimates were used in \cite{Basu}. 

Below we introduce some notation and formally state the main results.

\begin{definition}
For any $\pi, \tau \in S_n$, define the length of the longest common subsequence of $\pi$ and $\tau$ as follows,
\begin{align*}
\lcs(\pi, \tau) \coloneqq \max(m : \exists\,i_1 < \cdots <& i_m \text{ and } j_1 < \cdots < j_m \\
&\text{ such that } \pi(i_k) = \tau(j_k)\text{ for all } k \in [m]).
\end{align*}
\end{definition}

\begin{definition}
Given $\pi \in S_n$, the inversion set of $\pi$ is defined by
\[
\IV(\pi) \coloneqq \{ (i, j) : 1 \le i < j \le n \text{ and } \pi(i) > \pi(j) \},
\]
and the inversion number of $\pi$, denoted by $\inv(\pi)$, is defined to be the cardinality of $\IV(\pi)$.
\end{definition}
The Mallows measure on $S_n$ is introduced by Mallows in \cite{mallows1957}. For $q > 0$, the $(n, q)$ - Mallows measure on $S_n$ is given by
\[
\m(\pi) \coloneqq \frac{q^{\inv(\pi)}}{Z_{n, q}},
\]
where $Z_{n, q}$ is the normalizing constant. In other words, under the Mallows measure with parameter $q>0$, the probability of a permutation $\pi$ is proportional to $q^{\inv(\pi)}$.

The first result in this paper is the $L_p$ convergence of the length of the longest common subsequence of two independent Mallows permutations with same parameter $q_n$, such that $\lim_{n \to \infty} q_n = 1$ and $\lim_{n \to \infty}n(1-q_n) = \infty$.

\begin{theorem}\label{M6}
Suppose $\{q_n\}$ is a sequence such that
\[
0 < q_n < 1, \quad \lim_{n \to \infty} q_n = 1 \quad\text{ and }\quad \lim_{n \to \infty}n(1-q_n) = \infty.
\]
For each $n$, define two independent random variables $\pi_n, \tau_n$ such that $\pi_n \sim \n$ and $\tau_n \sim \n$. Then, for any $0 < p < \infty$,
\[
\frac{\lcs(\pi_n, \tau_n)}{n\sqrt{1 - q_n}} \overset{L_p}{\longrightarrow} \frac{\sqrt{6}}{3},
\]
as $n$ tends to infinity.
\end{theorem}

The proof of Theorem \ref{M6} follows the approach developed in \cite{Naya}, where the authors show a law of large numbers for the length of the longest increasing subsequence of Mallows permutation under a similar setting.

The second result in this paper is the following central limit theorem of the length of the LCS of two independent Mallows permutations with fixed parameters $q, q' \in (0, 1)$.
\begin{theorem}\label{M5}
Given $0 < q, q' < 1$, for each $n > 0$ define two independent random variables $\pi_n, \tau_n$ such that $\pi_n \sim \m$ and $\tau_n \sim \mu_{n, q'}$. There exist constant $\sigma = \sigma(q, q') > 0$ and $a = a(q, q') >0$ such that
\[
\frac{\lcs(\pi_n, \tau_n) - an}{\sigma \sqrt{n}} \overset{d}{\longrightarrow} \mathcal{N}(0, 1)
\]
as $n \to \infty$. Here $\overset{d}{\longrightarrow}$ denotes convergence in distribution and $\mathcal{N}(0, 1)$ denotes the standard Normal distribution.
\end{theorem}
The proof of Theorem \ref{M5} is based on the approach developed in \cite{Basu} in which Basu and Bhatnagar prove a central limit theorem of the length of the longest increasing subsequence of Mallows permutation with fixed parameter $q \in (0, 1)$.

\section{Proof of Theorem \ref{M6}}
\subsection{$q$-Mallows process}
In this section we describe a random process on permutations which was known to Mallows \cite{mallows1957}, and is termed as $q$-Mallows process in \cite{Naya}.
Given $q>0$, the $q$-Mallows process is a permutation-valued stochastic process $(p_n)_{n\ge1}$, where $p_n \in S_n$. The process is initialized by setting $p_1$ to be the only permutation on one element. The process iteratively constructs $p_n$ from $p_{n-1}$ and an independent random variable $p_n(n)$ distributed as a truncated geometric random variables. Precisely, let $\{p_n(n)\}_{n \ge 1}$ be a sequence of independent random variables with the distributions
\[
\p(p_n(n) = j) = \frac{q^{j-1}}{1 + q + \cdots + q^{n-1}} = \frac{(1-q)q^{j-1}}{1 - q^n}, \quad \forall 1 \le j \le n.
\]
Each permutation $p_n$ is defined iteratively by
\[
p_n(i) = \begin{cases}
p_{n-1}(i), &\text{when } p_{n-1}(i) < p_n(n);\\
p_{n-1}(i) + 1, &\text{when } p_{n-1}(i) \ge p_n(n);\\
p_n(n), &\text{when } $i = n$.
\end{cases}
\]
The $q$-Mallows process constructed as above has the following property (cf. Lemma 2.1 in \cite{Naya}).

\begin{lemma}\label{LE61}
Let $q > 0$ and let $\{p_n\}_{n \ge 1}$ be the $q$-Mallows process. Then $p_n$ is distributed according to $\mu_{n, 1/q}$.
\end{lemma}

The next lemma says that $p_i(i)$ is determined by the value of $p_n$ on $[i]$.

\begin{lemma}\label{LE611}
For any $1 \le i \le n$, we have
\begin{equation}\label{eq:611}
i - p_i(i) = \sum_{t = 1}^i \mathds{1}\left(p_n(t) > p_n(i)\right).
\end{equation}
\end{lemma}
\begin{proof}
By the definition of $q$-Mallows process, $p_i$ is a permutation in $S_i$. Hence we have
\begin{align*}
p_i(i) &= \sum_{t = 1}^i \mathds{1}\left(p_i(t) \le p_i(i)\right)\\
       &= \sum_{t = 1}^i \mathds{1}\left(p_n(t) \le p_n(i)\right)
\end{align*}
Here the last equality follows since the relative ordering of previous indices will not change by the following updates. Thus
\[
i - p_i(i) = \sum_{t = 1}^i 1 - \mathds{1}\left(p_n(t) \le p_n(i)\right) = \sum_{t = 1}^i \mathds{1}\left(p_n(t) > p_n(i)\right).
\]
\end{proof}

A direct corollary of Lemma \ref{LE611} is that the number of inversions of $p_n$ can be written as a function of $p_i(i)$.
\begin{corollary}\label{C611}
\begin{equation}
\inv(p_n) = \frac{(n+1)n}{2} - \sum_{i = 1}^n p_i(i)
\end{equation}
\end{corollary}

\begin{lemma}\label{LE612}
For any $1 \le i \le n$, we have
\begin{equation}\label{eq:612}
p_n(i) = p_i(i) + n - i - \sum_{t = i+1}^n \mathds{1}\left(p_n(t) > p_n(i)\right).
\end{equation}
Moreover, if $k \in [n]\setminus \{p_n(t): i+1 \le t \le n\}$ satisfies the following equation,
\begin{equation}\label{eq:6120}
k = p_i(i) + n - i - \sum_{t = i+1}^n \mathds{1}\left(p_n(t) > k\right),
\end{equation}
then we have $k = p_n(i)$.
\end{lemma}
\begin{proof}
Since $p_n$ is a permutation in $S_n$, we have
\[
p_n(i) = n - \sum_{t = 1}^n \mathds{1}\left(p_n(t) > p_n(i)\right).
\]
Hence (\ref{eq:612}) follows from (\ref{eq:611}). We prove the second claim by contradiction. Suppose we have $k < k'$ with $k, k' \in [n]\setminus \{p_n(t): i+1 \le t \le n\}$ such that
\begin{align*}
k &= p_i(i) + n - i - \sum_{t = i+1}^n \mathds{1}\left(p_n(t) > k\right),\\
k' &= p_i(i) + n - i - \sum_{t = i+1}^n \mathds{1}\left(p_n(t) > k'\right).
\end{align*}
By subtracting these two equations, we have
\begin{equation}\label{eq:contra}
k' - k = \sum_{t = i+1}^n \mathds{1}\left( k < p_n(t) \le k'\right) = \sum_{t = i+1}^n \mathds{1}\left( k < p_n(t) \le k' - 1\right),
\end{equation}
where the last equality follows since $k' \notin \{p_n(t): i+1 \le t \le n\}$. (\ref{eq:contra}) is a contradiction because $\{p_n(t): i+1 \le t \le n\}$ are distinct numbers and there are only $k - j - 1$ integers within $(j, k-1]$.
\end{proof}

\subsection{Basic properties of Mallows permutation}
In this section, we list a couple of properties of Mallows permutation. The proofs of the following lemmas can be found in Section 2 in \cite{Naya}.
\begin{definition}\label{DE61}
Given $\pi \in S_n$, let $\pi^r$ denote the reversal of $\pi$ which is defined by $\pi^r(i) = \pi(n+1-i)$. Let $\bs{a} = (a_1, \ldots, a_k)$ be an increasing sequence of indices in $[n]$. Define $\pi(\bs{a}) \coloneqq (\pi(a_1), \ldots, \pi(a_k))$. Let $\pi_{\bs{a}}$ denote the induced permutation in $S_k$ where $\pi_{\bs{a}}(i) = j$ if $\pi(a_i)$ is the $j$-th smallest term in $\pi(\bs{a})$.
\end{definition}

\begin{lemma}\label{LE62}
For any $q > 0$, if $\pi \sim \m$ then $\pi^r \sim \mu_{n, 1/q}$ and $\pi^{-1} \sim \m$.
\end{lemma}

\begin{lemma}\label{LE63}
Let $\bs{a} = (a_1, \ldots, a_k)$ and $\bs{b} = (b_1, \ldots, b_l)$ be two increasing sequences of indices in $[n]$ such that $a_k < b_1$. If $\pi \sim \m$, then $\pi_{\bs{a}}$ and $\pi(\bs{b})$ are independent and $\pi(\bs{a})$ and $\pi_{\bs{b}}$ are independent.
\end{lemma}

\begin{lemma}\label{LE64}
Let $\bs{I} = (i, i+1, \ldots, i + m -1) \subset [n]$ be a sequence of consecutive indices. If $\pi \sim \m$, then $\pi_{\bs{I}} \sim \mu_{m, q}$ and $\pi_{\pi^{-1}(\bs{I})} \sim \mu_{m, q}$. Moreover, conditioned on $\pi^{-1}(\bs{I}) = E \subset [n]$, we still have $\pi_{E} \sim \mu_{m, q}$.
\end{lemma}

\subsection{Reducing LCS problem to LIS problem}
\begin{definition}\label{D3}
Given a set of points in $\mathbb{R}^2$: $\bs{z} = \{z_1, z_2, \ldots, z_n\}$, where $z_i = (x_i, y_i) \in \mathbb{R}^2$, we say that $(z_{i_1}, z_{i_2}, \ldots, z_{i_m})$ is an increasing subsequence if
\[
x_{i_j} < x_{i_{j+1}}, \quad y_{i_j} < y_{i_{j+1}}, \quad j = 1, 2, \ldots, m-1.
\]
Here we do not require $i_j < i_{j+1}$.
Let $\lis(\bs{z})$ denote the length of the longest increasing subsequence of $\bs{z}$.
\end{definition}
\begin{definition}\label{D31}
Given $\bs{a} = (a_1, \ldots, a_n) \in \mathbb{R}^n$, $\bs{b} = (b_1, \ldots, b_n)~\in~\mathbb{R}^n$, we say that $((a_{i_1}, b_{i_1}), (a_{i_2}, b_{i_2}), \ldots, (a_{i_m}, b_{i_m}))$ is an increasing subsequence between $\bs{a}$ and $\bs{b}$ if
\[
a_{i_j} < a_{i_{j+1}}, \quad b_{i_j} < b_{i_{j+1}}, \quad j = 1, 2, \ldots, m-1
\]
Here we do not require $i_j < i_{j+1}$.
Let $\lis(\bs{a}, \bs{b})$ denote the length of the longest increasing subsequence between $\bs{a}$ and $\bs{b}$.
\end{definition}
Note that Definition \ref{D31} allows us to define $\lis(\pi,\tau)$, the length of the longest increasing subsequence of two permutations, by regarding $\pi$ and $\tau$ as vectors in $\mathbb{R}^n$.
In \cite{Ke2}, we show the following lemma which let us reduce the LCS problem to LIS problem.
\begin{lemma}\label{LE65}
For any $\pi, \tau \in S_n$, $\lcs(\pi, \tau) = \lis(\pi^{-1}, \tau^{-1})$.
\end{lemma}
The next lemma, also proved (as Lemma 3.9) in \cite{Ke2}, will be used to establish various inequalities directly from the results in \cite{Naya}. It says that the $\lis$ of two independent Mallows permutation restricted to a given collection of indices is dominated by the $\lis$ of a single Mallows permutation restricted to the same indices.
\begin{lemma}\label{L360}
Given $\bs{a} = (a_1, a_2, \ldots, a_k)$, where $a_1 < \cdots < a_k$ and $a_i \in[n]$, for any $0 < q \le 1$ and any distribution $\nu$ on $S_k$, there exists a coupling $(X, Y, Z)$ such that the following holds,
\begin{itemize}
\item[(a)] $X$ and $Y$ are independent.
\item[(b)] $X \sim \m$, $Y \sim \nu$ and $Z \sim \m$.
\item[(c)] $\lis(X_{\bs{a}}, Y) \le \lis(Z_{\bs{a}})$.
\end{itemize}
\end{lemma}

A special case of Lemma \ref{L360} is when we choose $\bs{a} = (1, 2, \ldots, n)$.  A direct consequence of Lemma \ref{L360} is that we can obtain some large deviation bounds for the LCS of two independent permutations at least one of which is Mallows distributed from the large deviation bounds for the LIS of a single Mallows permutation.

By Lemma \ref{L360}, for any $n \ge 1$, there exists a coupling $(\pi_n, \tau_n, Z_n)$ such that $\pi_n, \tau_n$ and $Z_n$ are all $\m$-distributed with $\pi_n, \tau_n$ being independent and
\begin{equation}\label{eq:699}
\lis(\pi_n, \tau_n) \le \lis(Z_n).
\end{equation}
In \cite{Naya} Section 5.1, the authors show that, given $p > 0$, when $q$ is sufficiently close to $1$, the family of random variables $\left\{\left|\frac{\lis(Z_n)}{n\sqrt{1-q}}\right|^p\right\}$ indexed by $q$ is uniformly integrable. Hence by (\ref{eq:699}), the family of random variables $\left\{\left|\frac{\lis(\pi_n, \tau_n)}{n\sqrt{1-q}}\right|^p\right\}$ is also uniformly integrable.
In the following we show that
\begin{equation}\label{eq:61}
\frac{\lis(\pi_n, \tau_n)}{n\sqrt{1 - q_n}} \overset{L_1}{\longrightarrow} \frac{\sqrt{6}}{3},
\end{equation}
as $n \to \infty$. Then, by the uniform integrability of $\left\{\left|\frac{\lis(\pi_n, \tau_n)}{n\sqrt{1-q}}\right|^p\right\}$, for any $p > 0$, we have
\begin{equation}\label{eq:613}
\frac{\lis(\pi_n, \tau_n)}{n\sqrt{1 - q_n}} \overset{L_p}{\longrightarrow} \frac{\sqrt{6}}{3},
\end{equation}
as $n \to \infty$. Therefore Theorem \ref{M6} follows from Lemma \ref{LE65} and the fact that $(\pi_n, \tau_n)$ has the same distribution as $(\pi^{-1}_n, \tau^{-1}_n)$.
The proof of (\ref{eq:61}) follows the approach developed in \cite{Naya} in which the authors prove a similar result for the length of the longest increasing subsequence of Mallows permutation.

\subsection{Block decomposition}\label{SE:bd}
Let $n = n(q)$ be a function of $q$ such that
\begin{equation}\label{eq:614}
\lim_{q \to 1} n = \infty,\quad \text{ and } \quad\lim_{q \to 1}n(1-q) = \infty.
\end{equation}
Let $\pi \sim \m$, $\tau \sim \m$ and $\pi$ and $\tau$ are independent. To prove (\ref{eq:61}), it suffices to show that
\begin{equation}\label{eq:64}
\frac{\lis(\pi, \tau)}{n\sqrt{1 - q}} \overset{L_1}{\longrightarrow} \frac{\sqrt{6}}{3},
\end{equation}
as $q \to 1$. In the following, we will partition $[n]$ into blocks of size $\frac{\beta}{1 - q}$ for some large $\beta$. Considering $\lis(\pi, \tau)$ when restricting $\pi$ and $\tau$ in each blocks, we will show that the concatenation of these increasing subsequences within each block is close to $\lis(\pi, \tau)$.

Given $\beta > 0$, define a function $\beta(q)$ such that $\frac{\beta(q)}{1-q}$ is an integer and $\beta(q) \to \beta$ as $q \to 1$. Define
\begin{equation}\label{eq:605}
m \coloneqq \left\lfloor \frac{n(1-q)}{\beta(q)}\right\rfloor.
\end{equation}
For $1 \le i \le m$ define
\[
B_i \coloneqq \left((i-1)\frac{\beta(q)}{1-q} + 1, \ldots, i \frac{\beta(q)}{1-q} \right).
\]
Hence, each $B_i$ is a block of consecutive integers of size $\frac{\beta(q)}{1-q}$. To make $\{B_i\}$ a partition of $[n]$, define $B_{m+1} \coloneqq \left(m \frac{\beta(q)}{1-q} + 1, \ldots, n\right)$. For $1 \le i \le m+1$, let
\[
X_i \coloneqq \lis(\pi_{B_i}, \tau_{B_i})
\]
be the LIS of the restriction of $\pi$ and $\tau$ to $B_i$ as defined in Definition \ref{D31}. By Lemma \ref{LE63}, the $X_i$ are independent. By Lemma \ref{LE64}, each $X_i$ has the distribution of the LIS of two independent Mallows permutations of size $\frac{\beta(q)}{1-q}$ and parameter $q$. Moreover, by Lemma \ref{LE62}, and using Lemma \ref{LE65}, $X_i$ has the distribution of the LCS of two independent Mallows permutations of size $\frac{\beta(q)}{1-q}$ and parameter $q$.
By the triangle inequality, we have
\[
\left|\frac{\lis(\pi, \tau)}{n\sqrt{1 - q}} - \frac{\sqrt{6}}{3}\right| \le 
\left|\frac{\lis(\pi, \tau) - \sum_{i = 1}^m X_i}{n\sqrt{1 - q}}\right| + 
\left|\frac{\sum_{i = 1}^m X_i}{n\sqrt{1 - q}} - \frac{\sqrt{6}}{3}\right|.
\]
We will prove that
\begin{equation}\label{eq:62}
\varlimsup_{\beta \to \infty}\varlimsup_{q \to 1} \E\left(\left|\frac{\lis(\pi, \tau) - \sum_{i = 1}^m X_i}{n\sqrt{1 - q}}\right|\right) = 0,
\end{equation}

\begin{equation}\label{eq:63}
\varlimsup_{\beta \to \infty}\varlimsup_{q \to 1} \E\left(\left|\frac{\sum_{i = 1}^m X_i}{n\sqrt{1 - q}} - \frac{\sqrt{6}}{3}\right|\right) = 0.
\end{equation}

These equalities imply that
\[
\varlimsup_{\beta \to \infty}\varlimsup_{q \to 1} \E\left(\left|\frac{\lis(\pi, \tau)}{n\sqrt{1 - q}} - \frac{\sqrt{6}}{3}\right|\right) = 0,
\]
and since $\pi$ and $\tau$ do not depend on $\beta$, we have
\[
\lim_{q \to 1} \E\left(\left|\frac{\lis(\pi, \tau)}{n\sqrt{1 - q}} - \frac{\sqrt{6}}{3}\right|\right) = 0
\]
which is exactly (\ref{eq:64}).

\subsection{Comparing $\lis(\pi, \tau)$ and $\sum X_i$}
Since $\{B_i\}$ partition $[n]$, it follows trivially that
\begin{equation}\label{eq:65}
\lis(\pi, \tau) \le \sum_{i= 1}^{m+1} X_i.
\end{equation}
We will show a bound in the other direction by using the $q$-Mallows process. Given two independent $q$-Mallows processes $\{p_i\}$ and $\{p'_i\}$, define two permutations $\pi$ and $\tau$ by
\begin{equation}\label{eq:651}
\pi(j)\coloneqq n + 1 - p_n(j),  \qquad \tau(j)\coloneqq n + 1 - p'_n(j),
\end{equation}
for $1 \le j \le n$. By Lemma \ref{LE61} and Lemma \ref{LE62}, it follows that $\pi \sim \m$ and $\tau \sim \m$. Let $a = a(\beta) > 0$ be any function of $\beta$ satisfying
\begin{equation}\label{eq:606}
a \to \infty \quad \text{ and } \quad \frac{a}{\beta} \to 0, \quad \text{ as } \beta \to \infty.
\end{equation}
For each $i \in [m]$ define
\[
E_i \coloneqq \left\{j \in B_i : p_{\max{B_i}}(j) \le \frac{a}{1-q}\right\}, \quad
F_i \coloneqq \left\{j \in B_i : p_j(j) > \frac{a}{1-q} \right\}.
\]
That is $E_i$ consists of those indices in $B_i$ at which the first $q$-Mallows process is at most $\frac{a}{1-q}$ after the entire block $B_i$ is assigned. $F_i$ consists of those indices in $B_i$ at which its initial position is greater than $\frac{a}{1-q}$. For the second $q$-Mallows process, we define $E'_i$ and $F'_i$ similarly.

Let $I_i = (i_1, \ldots, i_k) \subset B_i$ be the indices of an arbitrary longest increasing subsequence of $\pi$ and $\tau$ in the restriction of $B_i$. That is $\pi(i_j) < \pi(i_{j+1})$ and $\tau(i_j) < \tau(i_{j+1})$ for any $j \in [k-1]$. Note that by the definition of $X_i$, we have $|I_i| = X_i$. Define
\[
I'_i \coloneqq I_i \setminus (E_i \cup F_i \cup E'_i \cup F'_i).
\]
In other words, $I'_i$ is obtained by delete those indices in $E_i \cup F_i \cup E'_i \cup F'_i$ from $I_i$ without changing the ordering of the remaining indices in $I_i$. The definitions of $B_i$, $E_i$, $F_i$, $E'_i$ and $F'_i$ imply that the concatenation of $\{I'_i\}_{i \in [m]}$ is a set of indices along which defines an increasing subsequence of $\pi$ and $\tau$. To see this, suppose $j, k$ come from the same $I'_i$ with $j$ comes before $k$ in $I'_i$, then by the definition of $I'_i$ we have $\pi(j) < \pi(k)$ and $\tau(j) < \tau(k)$. On the other hand, suppose $j \in I'_{s}$ and $k \in I'_{t}$ with $s < t$. By the definition of $E_{s}$, $F_{t}$, we have
\[
p_k(k) \le \frac{a}{1-q} < p_{\max{B_{s}}}(j) \le p_k(j),
\]
which implies that $p_n(k) < p_n(j)$, thus $\pi(k) > \pi(j)$. The inequality $\tau(k) > \tau(j)$ follows from the similar argument. Hence
\begin{equation}\label{eq:66}
\lis(\pi, \tau) \ge \sum_{i = 1}^{m} |I'_i|
\end{equation}
Moreover, the definitions of $I_i$ and $I'_i$ imply that
\begin{equation}\label{eq:67}
X_i = |I_i| \le |I'_i| +  \sum_{A \in \{E_i, E'_i, F_i, F'_i\}}\lis(\pi_{A}, \tau_{A})
\end{equation}
for $1 \le i \le m$. From (\ref{eq:66}) and (\ref{eq:67}), we have
\begin{equation}\label{eq:68}
\lis(\pi, \tau) \ge \sum_{i = 1}^m X_i - \sum_{A \in \{E_i, E'_i, F_i, F'_i\}}\lis(\pi_{A}, \tau_{A}).
\end{equation}
Thus from (\ref{eq:65}) and (\ref{eq:68}), we get
\begin{equation}\label{eq:69}
\E\bigg(\bigg|\lis(\pi, \tau) - \sum_{i = 1}^m X_i\bigg|\bigg) \le \sum_{i = 1}^m \sum_{A \in \{E_i, E'_i, F_i, F'_i\}}\E\left(\lis(\pi_{A}, \tau_{A})\right) + \E(X_{m+1}).
\end{equation}
Therefore, (\ref{eq:62}) is a direct consequence of the next lemma.

\begin{lemma}\label{LE66}
\begin{equation}\label{eq:601}
\varlimsup_{\beta \to \infty}\varlimsup_{q \to 1} \E\left(\frac{X_{m+1}}{n\sqrt{1-q}}\right) = 0.
\end{equation}
\begin{equation}\label{eq:602}
\varlimsup_{\beta \to \infty}\varlimsup_{q \to 1} \frac{\sum_{i = 1}^m \E(\lis(\pi_{A_i}, \tau_{A_i}))}{n\sqrt{1-q}} = 0,
\end{equation}
for $A_i \in \{E_i, E'_i, F_i, F'_i\}$.
\end{lemma}

Before proving Lemma \ref{LE66}, we state the following technical lemma whose proof will be presented at the end of this section. Both Lemma \ref{LE67} and Lemma \ref{L360} will be used to reduce the claim in Lemma \ref{LE66} to the result of Lemma 5.1 in \cite{Naya}.

\begin{lemma}\label{LE67}
Given consecutive indices $B \subset [n]$, $0 < q < 1$ and any constant $C > 0$, there exists a coupling of $q$-Mallows processes $\{\bar{p}_i\}, \{p'_{i}\}$ and $\{\hat{p}_i\}$ such that
\begin{itemize}
\item $\{\bar{p}_i\}$ and $\{p'_{i}\}$ are independent.
\item Define $\pi(j)\coloneqq n + 1 - \bar{p}_n(j)$, $\tau(j)\coloneqq n + 1 - p'_n(j)$, $\hat{\pi}(j) \coloneqq n+1 - \hat{p}_n(j)$ and
\[
\bar{F} \coloneqq \left\{ j \in B : \bar{p}_j(j) > C \right\}, \quad \hat{F} \coloneqq \left\{ j \in B : \hat{p}_j(j) > C \right\}.
\]
Then, we have $\bar{F} = \hat{F}$ and $\lis(\pi_{\bar{F}}, \tau_{\bar{F}}) \le \lis(\hat{\pi}_{\bar{F}})$.
\end{itemize}
\end{lemma}

\begin{proof}[Proof of Lemma \ref{LE66}]
To show (\ref{eq:601}), we define $X$ to be a random variable which has the same distribution as $\lis(\pi_{B_{m+1}})$. By Lemma~\ref{L360}, letting $\bs{a} = B_{m+1}$, we have
\[
\E(X_{m+1}) = \E(\lis(\pi_{B_{m+1}}, \tau_{B_{m+1}})) \le \E(X),
\]
and (\ref{eq:601}) follows from the first equation in Lemma 5.1 in \cite{Naya}.

To prove (\ref{eq:602}), by symmetry, we only need to show (\ref{eq:602}) holds when $A_i = E_i, F_i$. For the case when $A_i = E_i$, define
\begin{equation}\label{eq:603}
I \coloneqq \textstyle\big(1, 2, \ldots, \big\lfloor\frac{a}{1-q}\big\rfloor \big), \quad \sigma \coloneqq p_{\max{B_i}}, \quad \sigma' \coloneqq p'_{\max{B_i}}, \quad \bar{E_i} \coloneqq \sigma^{-1}(I).
\end{equation}
We have
\begin{align}
\lis(\pi_{E_i}, \tau_{E_i}) \le \lis(\pi_{\bar{E}_i}, \tau_{\bar{E}_i}) &= \lis((p_n)_{\bar{E}_i}, (p'_n)_{\bar{E}_i}) \nonumber \\
 &= \lis(\sigma_{\bar{E}_i}, \sigma'_{\bar{E}_i}) \nonumber \\
 & = \lis((\sigma_{\bar{E}_i})^r, (\sigma'_{\bar{E}_i})^r) \label{eq:604}
\end{align}
By Lemma \ref{LE64} and (\ref{eq:603}), conditioned on the value of $\bar{E}_i$, we have $\sigma_{\bar{E}_i} \sim \mu_{\left\lfloor\frac{a}{1-q}\right\rfloor, 1/q}$. By Lemma \ref{LE62}, we have $(\sigma_{\bar{E}_i})^r \sim \mu_{\left\lfloor\frac{a}{1-q}\right\rfloor, q}$. Moreover, conditioned on the value of $\bar{E}_i$, $\,(\sigma_{\bar{E}_i})^r$ and $(\sigma'_{\bar{E}_i})^r$ are independent. Thus, by choosing $\bs{a} = I$ in Lemma \ref{L360}, there exists a random variable $Z$ with $Z \sim \mu_{\left\lfloor\frac{a}{1-q}\right\rfloor, q}$ such that
\[
\lis((\sigma_{\bar{E}_i})^r, (\sigma'_{\bar{E}_i})^r) \le \lis(Z).
\]
Hence it follows from (\ref{eq:604}) that $\lis(\pi_{E_i}, \tau_{E_i}) \le \lis(Z)$. For any $a > 5$, since $0 < q < 1$, we have $\big\lfloor\frac{a}{1-q}\big\rfloor > 5$. Thus
\[
1 - \frac{4}{\big\lfloor\frac{a}{1-q}\big\rfloor} \ge 1 - \frac{5}{\frac{a}{1-q}} > q
\]
Hence, by Theorem 1.3 in \cite{Naya}, there exists a constant $c$ such that
\[
\E(\lis(\pi_{E_i}, \tau_{E_i})) \le \E(\lis(Z)) \le c\left\lfloor\frac{a}{1-q}\right\rfloor\sqrt{1 - q} \le \frac{ca}{\sqrt{1-q}}.
\]
Hence, from the definition of $m$ in (\ref{eq:605}) and the property of $a$ as defined in (\ref{eq:606}), it follows that
\begin{align*}
& \varlimsup_{\beta \to \infty}\varlimsup_{q \to 1} \frac{\sum_{i = 1}^m \E(\lis(\pi_{E_i}, \tau_{E_i}))}{n\sqrt{1-q}} 
\le \varlimsup_{\beta \to \infty}\varlimsup_{q \to 1} \frac{mca}{n(1-q)}\\
\le\,& \varlimsup_{\beta \to \infty}\varlimsup_{q \to 1}\frac{ca}{\beta(q)} = \varlimsup_{\beta \to \infty} \frac{ca}{\beta} = 0,
\end{align*}
which completes the proof of (\ref{eq:602}) when $A_i = E_i$. For the case when $A_i = F_i$, by Lemma \ref{LE67}, there exists a coupling such that
\begin{equation}
\E(\lis(\pi_{F_i}, \tau_{F_i})) \le \E(\lis(\hat{\pi}_{F_i}))
\end{equation}
The claim follows directly from the third equation in Lemma 5.1 in \cite{Naya}.
\end{proof}

Next we establish (\ref{eq:63}), which combined with (\ref{eq:62}) implies (\ref{eq:64}), which completes the proof of Theorem \ref{M6}. We rely on the following result in \cite{Ke2}, in which a weak law of large numbers of the $\lcs$ of two independent Mallows permutations is established in the regime where $n(1-q)$ has finite limit as $n$ tends to infinity.

\begin{theorem}\label{M3}
Suppose that $\{q_n\}$ is a sequence such that $\lim_{n \to \infty}n(1 - q_n) = \beta \in \mathbb{R}$. Define independent Mallows permutations $\pi_n \sim \n$ and $\tau_n \sim \n$. For any $\epsilon > 0$, we have
\[
\lim_{n\to\infty}\p\left(\left|\frac{\lcs(\pi_n, \tau_n)}{\sqrt{n}} - 2\bar{J}(\beta)\right| < \epsilon \right) = 1,
\]
where
\begin{equation} \label{eq:m3}
\bar{J}(\beta) = {\textstyle\sqrt{\frac{\beta}{6 \sinh{(\beta/2)}}}}\cdot\int_0^1 \sqrt{ \cosh{(\beta/2)} + 2 \cosh{\big(\beta[2x - 1]/2\big)}}\, dx.
\end{equation}
\end{theorem}
First we show that
\begin{equation}\label{eq:697}
\lim_{\beta \to \infty} \frac{\bar{J}(\beta)}{\sqrt{\beta}} = \frac{1}{\sqrt{6}}.
\end{equation}
Since $\lim_{x \to \infty} \coth(x) = 1$, by (\ref{eq:m3}), it suffices to show
\begin{equation}\label{eq:696}
\lim_{\beta \to \infty}\int_0^1 \sqrt{ 1 + 2 \cosh{\big(\beta[2x - 1]/2\big)}\big/\cosh{(\beta/2)}}\, dx = 1
\end{equation}
Note that
\begin{align*}
  & 1 + 2\cdot \frac{\cosh{\big(\beta[2x - 1]/2\big)}}{\cosh{(\beta/2)}}\\
=\,  &1 + 2 \cdot \frac{e^{\beta(2x-1)/2} + e^{-\beta(2x-1)/2}}{e^{\beta/2} + e^{-\beta/2}}\\
=\,& 1 + 2 \cdot \frac{e^{\beta(x-1)} + e^{-\beta x}}{1 + e^{-\beta}} < 1 + 2(1 + 1) = 5
\end{align*}
for any $x \in [0, 1]$ and $\beta > 0$. Hence, by dominated convergence theorem, we have
\begin{align*}
& \lim_{\beta \to \infty}\int_0^1 \sqrt{ 1 + 2 \cosh{\big(\beta[2x - 1]/2\big)}\big/\cosh{(\beta/2)}}\, dx\\
=\, & \int_0^1 \lim_{\beta \to \infty}\sqrt{ 1 + 2 \cosh{\big(\beta[2x - 1]/2\big)}\big/\cosh{(\beta/2)}}\, dx\\
=\, & \int_0^1 1 \,dx = 1.
\end{align*}
(\ref{eq:696}) as well as (\ref{eq:697}) follow.

We continue with the notation defined in Section \ref{SE:bd}. Suppose $n = n(q)$ is such that (\ref{eq:614}) holds. Recall that $X_1$ denotes the length of the LIS of two independent Mallows permutations with the same distribution $\mu_{\frac{\beta(q)}{(1-q)}, q}$. Since
\[
\lim_{q \to 1}\frac{\beta(q)}{1-q}\cdot(1-q) = \beta,
\]
we can apply Theorem \ref{M3} and Lemma \ref{LE65} to $X_1$ and deduce that
\begin{equation}\label{eq:695}
\sqrt{\frac{1 - q}{\beta(q)}}\cdot X_1 \overset{p}{\longrightarrow} 2\bar{J}(\beta).
\end{equation}
Now fix $\beta_0$ sufficiently large and $q_0$ sufficiently close to $1$ such that $\beta > \beta_0$ and $q_0 \le q < 1$ imply $\frac{1}{2} < q < 1 - \frac{4(1-q)}{\beta(q)}$. By (68) in \cite{Naya} and Lemma \ref{L360}, it follows that
\begin{equation}\label{eq:694}
\left\{\left(\frac{\sqrt{1-q}}{\beta(q)}\cdot X_1\right)^2\right\} \text{ indexed by } q_0 < q < 1 \text{ are uniformly integrable.}
\end{equation}
Since $\beta(q) \to \beta$ as $q \to 1$, (\ref{eq:695}) and (\ref{eq:694}) imply that for any fixed $\beta > \beta_0$,
\[
\sqrt{\frac{1-q}{\beta}}\cdot X_1 \overset{L_2}{\longrightarrow} 2\bar{J}(\beta),
\]
as $q \to 1$. Hence, for any fixed $\beta > \beta_0$, we have
\begin{equation}\label{eq:693}
\lim_{q \to 1} \sqrt{\frac{1-q}{\beta}}\cdot \E(X_1) = 2\bar{J}(\beta) \quad\text{ and }\quad \lim_{q \to 1}( 1-q)\cdot \var(X_1) = 0.
\end{equation}
Let $Y \coloneqq \frac{\sum_{i = 1}^m X_i}{n \sqrt{1-q}}$. To prove (\ref{eq:63}), we first show that
\begin{align}
&\lim_{\beta \to \infty} \lim_{q \to 1} \E(Y) = \frac{\sqrt{6}}{3}, \label{eq:692} \\
&\lim_{\beta \to \infty} \lim_{q \to 1} \var(Y) = 0. \label{eq:691}
\end{align}
To prove (\ref{eq:692}), note that since $\{X_i\}_{i\in [m]}$ are i.i.d.\,random variables, we have
\begin{equation}\label{eq:690}
\E(Y) = \frac{m}{n\sqrt{1-q}} \E(X_1) = \frac{m\beta}{n(1-q)}\cdot\frac{\sqrt{1-q}}{\beta}\cdot \E(X_1).
\end{equation}
By the definition of $m$ and (\ref{eq:614}), we have
\begin{equation}\label{eq:688}
\lim_{q \to 1} \frac{m\beta}{n(1-q)} = 1.
\end{equation}
Hence, from (\ref{eq:690}) and using (\ref{eq:693}), it follows that
\begin{equation}\label{eq:689}
\lim_{q \to 1}\E(Y) = \frac{1}{\sqrt{\beta}}\cdot \lim_{q \to 1} \sqrt{\frac{1-q}{\beta}}\cdot \E(X_1) = \frac{2 \bar{J}(\beta)}{\sqrt{\beta}}.
\end{equation}
Thus, (\ref{eq:692}) follows from (\ref{eq:697}), since
\[
\lim_{\beta \to \infty} \lim_{q \to 1} \E(Y) = \lim_{\beta \to \infty}\frac{2 \bar{J}(\beta)}{\sqrt{\beta}} = \frac{\sqrt{6}}{3}.
\]
To prove (\ref{eq:691}), again since $\{X_i\}_{i\in [m]}$ are i.i.d., by (\ref{eq:688}), we have
\begin{align*}
\lim_{q \to 1}\var(Y) &= \lim_{q \to 1}\frac{m}{n^2(1-q)}\var(X_1)\\
                      &= \lim_{q \to 1}\frac{1}{\beta n}\var(X_1) = \lim_{q \to 1}\frac{1}{\beta n(1-q)}(1-q)\var(X_1).
\end{align*}
Hence, for $\beta > \beta_0$, (\ref{eq:614}) and (\ref{eq:693}) imply that
\[
\lim_{q \to 1} \var(Y) = 0,
\]
proving (\ref{eq:691}).
Finally, by the triangle and Jensen's inequalities we have
\[
\textstyle\E\left|Y - \frac{\sqrt{6}}{3}\right| \le \E\Big|Y - \E(Y)\Big| + \left|\E(Y) - \frac{\sqrt{6}}{3}\right| \le \sqrt{\var(Y)} + \left|\E(Y) - \frac{\sqrt{6}}{3}\right|,
\]
which shows that (\ref{eq:692}) and (\ref{eq:691}) imply (\ref{eq:63}).

\subsection{Proof of Lemma \ref{LE67}}
The proof of Lemma \ref{LE67} is by induction on the number of inversions of~$\tau$. In the following, we establish the induction step in Claim \ref{claim2}. First we will prove the following claim. 
\begin{claim} \label{claim1}
Let $\{p_i\}$ be a $q$-Mallows process. Given a block of consecutive indices $B$ and any positive constant $C$, let $M \coloneqq  \max\{i \in B\}$ and define $F \coloneqq \left\{ j \in B : p_j(j) > C \right\}$. Given increasing indices $\bs{v} = (v_1, v_2, \ldots, v_l)$ with $v_i \in B$, for any $1 \le j < k \le l$ and any permutation $\bs{b} = (b_1, b_2, \ldots, b_M) \in S_M$ with $b_{v_j} < b_{v_k}$, we have
\[
\p\left(p_M = \bs{b}\,\big|\, F = \bs{v}\right) \le \p\left(p_M = \bs{b}\ccirc(v_j, v_k)\,\big|\, F = \bs{v}\right).
\]
Here $\bs{b}\ccirc(v_j, v_k)$ denotes the permutation obtained by switching $b_{v_j}$ and $b_{v_k}$ in $\bs{b}$. We abuse the notation $F = \bs{v}$ to indicate that the set of the elements in vector $\bs{v}$ is equal to $F$. 
\end{claim}

\begin{proof}[Proof of Claim \ref{claim1}]
If $\p\left(p_M = \bs{b}\,\big|\, F = \bs{v}\right) = 0$, the claim holds trivially. Suppose $\p\left(p_M = \bs{b}\,\big|\, F = \bs{v}\right) > 0$, i.e.\,there exists $\bs{t} = (t_1, \ldots, t_M)$ such that
\begin{enumerate}[label=(\roman*)]
\item\label{eq:657} $1 \le t_i \le i$, 
\item\label{eq:658} for $i \in B$, $t_i > C$ if only if $i \in \bs{v}$,
\item\label{eq:659} if $p_i(i) = t_i$ for $i \in [M]$, we have $p_M = \bs{b}$.
\end{enumerate}
Define
\begin{equation}\label{eq:660}
\hat{t}_i \coloneqq
\begin{cases}
 t_i &\text{if $1 \le i < v_j$ or $v_k < i \le M$};\\
 t_i - \mathds{1}\left(b_{v_j} < p_M(i) < b_{v_k}\right) &\text{if $v_j < i < v_k$};\\
 v_j - \sum_{i = 1}^{v_j}\mathds{1}\left(p_M(i) > b_{v_k}\right) &\text{if $i = v_j$};\\
 v_k - \sum_{i = 1}^{v_k}\mathds{1}\left(p_M(i) > b_{v_j}\right) &\text{if $i = v_k$}.
\end{cases}
\end{equation}
We show that, if at each step of the $q$-Mallows process $\{\hat{p}_i\}$,
\begin{equation}\label{eq:661}
\hat{p}_{i}(i) = \hat{t}_i \quad \text{ for any } i \in[M]
\end{equation}
we have $\hat{p}_M = \bs{b} \ccirc (v_j, v_k)$. Moreover, if we define $\hat{F} \coloneqq \left\{ i \in B : \hat{p}_i(i) > C \right\}$, then $F = \hat{F}$.

We first show that $\hat{t}_{v_i}$ as defined in (\ref{eq:660}) satisfy that $C < \hat{t}_{v_i} \le v_i$, which implies that $F \subset \hat{F}$. We will prove this claim in different cases depending on the value of $i$.
\begin{itemize}
\item For $1 \le i < j$ or $k < i \le l$, we have $\hat{t}_{v_i} = t_{v_i}$. Thus by \ref{eq:657} and \ref{eq:658}, it follows that $C < \hat{t}_{v_i} \le v_i$. 
\item For $j < i < k$, we have
\[
\hat{t}_{v_i} \le t_{v_i} \le v_i.
\]
On the other hand, by the definition of $q$-Mallows process, $p_M(v_i) > b_{v_j}$ if and only if $p_{v_i}(v_i) > p_{v_i}(v_j)$. Hence if $\mathds{1}\left(b_{v_j} < p_M(v_i) < b_{v_k}\right) = 1$, we have 
\[
t_{v_i} = p_{v_i}(v_i) > p_{v_i}(v_j) \ge t_{v_j} > C,
\]
which means  $\mathds{1}\left(b_{v_j} < p_M(v_i) < b_{v_k}\right) = 1$ implies $t_{v_i} > C + 1$. Thus 
\[
\hat{t}_{v_i} = t_{v_i} - \mathds{1}\left(b_{v_j} < p_M(v_i) < b_{v_k}\right) > C.
\]
\item To show $C < \hat{t}_{v_j} \le v_j$, note that by the definition of $\hat{t}_{v_j}$ in (\ref{eq:660}), we have $\hat{t}_{v_j} \le v_j$. To show $\hat{t}_{v_j} > C$, note that since $p_{v_j}$ is a permutation in $S_{v_j}$, we have
\begin{align}
v_j - t_{v_j} &= \sum_{i = 1}^{v_j}\mathds{1}\left(p_{v_j}(i) > t_{v_j}\right) \nonumber\\
& = \sum_{i = 1}^{v_j}\mathds{1}\left(p_M(i) > b_{v_j}\right) \ge \sum_{i = 1}^{v_j}\mathds{1}\left(p_M(i) > b_{v_k}\right). \label{eq:668}
\end{align}
Here the last inequality follows since $b_{v_j} < b_{v_k}$. The definition of $\hat{t}_{v_j}$ and (\ref{eq:668}) imply $\hat{t}_{v_j} \ge t_{v_j} > C$.
\item To show $C < \hat{t}_{v_k} \le v_k$, again by the definition of $\hat{t}_{v_k}$ in (\ref{eq:660}), we have $\hat{t}_{v_k} \le v_k$. To show $\hat{t}_{v_k} > C$, note that since $p_{v_k}$ is a permutation in $S_{v_k}$, we have
\begin{align}
v_k - t_{v_j} &= \sum_{i = 1}^{v_k}\mathds{1}\left(p_{v_k}(i) > t_{v_j}\right) \nonumber\\
&\ge \sum_{i = 1}^{v_k}\mathds{1}\left(p_{v_k}(i) > p_{v_k}(v_j)\right) = \sum_{i = 1}^{v_k}\mathds{1}\left(p_M(i) > b_{v_j}\right)  \label{eq:675}
\end{align}
Here the inequality follows since $t_{v_j} = p_{v_j}(v_j) \le p_{v_k}(v_j)$. The definition of $\hat{t}_{v_k}$ and (\ref{eq:675}) imply $\hat{t}_{v_k} \ge t_{v_j}> C$
\end{itemize}

To show $\hat{F} \subset F$, note that for $i \in B \setminus \bs{v}$, by the definition of $\hat{t}_i$, we have $\hat{t}_i \le C$. For $v_j < i < v_k$, $\mathds{1}\left(b_{v_j} < p_M(i) < b_{v_k}\right) = 1$ implies $t_i = p_i(i) > p_i(v_j) \ge 1$. Hence
\[
\hat{t}_i = t_i - \mathds{1}\left(b_{v_j} < p_M(i) < b_{v_k}\right) \ge 1.
\]
Since $v_j < v_k$ and $b_{v_j} < b_{v_k}$, it follows from the definition of $\hat{t}_{v_j}$ and $\hat{t}_{v_k}$ that both of them are greater than 0. Therefore, we have shown $F = \hat{F}$.
The fact that $\hat{p}_{i}(i) = \hat{t}_i$ at every step $i \in [M]$ implies $\hat{p}_M = \bs{b} \ccirc (v_j, v_k)$ can be proved by induction. The induction is taken in reverse order with the base case $i = M$ and the induction step is established by using the second part of Lemma \ref{LE612} and the definition of $\hat{t}_{i}$. Specifically, by Lemma \ref{LE612}, for any $i \in M$,
\begin{equation}\label{eq:700}
p_M(i) = t_i + M - i - \sum_{r = i + 1}^{M}\mathds{1}\left(p_M(r) > p_M(i)\right).
\end{equation}
If $v_k < M$, by the definition of $\hat{t}_i$ and the fact that the value of $p_M(i)$ is determined by $\{t_j : i \le j \le M\}$, it follows that $p_M(i) = \hat{p}_M(i)$ for $i > v_k$. If $v_k = M$, then by (\ref{eq:660}), we have
\[
\hat{t}_M = M - \sum_{i = 1}^M \mathds{1}\left(p_M(i) > b_{v_j}\right) = b_{v_j}.
\]
Here the last equality follows since $p_M$ is a permutation in $S_M$. Hence we have $\hat{p}_M(v_k) = \hat{p}_M(M) = \hat{t}_M = b_{v_j}$. On the other hand if $v_k < M$, we plug in $k = b_{v_j}$ and $\hat{p}_{v_k}(v_k) = \hat{t}_{v_k}$ to (\ref{eq:6120}) and verify that the equality holds. The right hand side of (\ref{eq:6120} becomes
\begin{align*}
&\hat{t}_{v_k} + M - v_k - \sum_{r = v_k + 1}^M \mathds{1}\left(\hat{p}_M(r) > b_{v_j} \right)\\
=\,& v_k - \sum_{r = 1}^{v_k} \mathds{1}\left(p_M(r) > b_{v_j} \right) + M - v_k - \sum_{r = v_k + 1}^M \mathds{1}\left(\hat{p}_M(r) > b_{v_j} \right)\\
=\,& M - \sum_{r = 1}^{M} \mathds{1}\left(p_M(r) > b_{v_j} \right)\\
=\,& b_{v_j}.
\end{align*}
Here the second equality follows since by the induction hypothesis $p_M(i) = \hat{p}_M(i)$ for $i > v_k$, and the last equality follows since $p_M$ is a permutation in $S_M$. Hence by Lemma \ref{LE612}, $\hat{p}_M(v_k) = b_{v_j}$.

Next, if $v_j < i < v_k$, we have
\begin{equation}\label{eq:702}
\sum_{r = i + 1}^{M}\mathds{1}\left(p_M(r) > p_M(i)\right) = \sum_{r = i + 1}^{M}\mathds{1}\left(\hat{p}_M(r) > p_M(i)\right) + \mathds{1}\left(b_{v_j} < p_M(i) < b_{v_k}\right).
\end{equation}
Indeed, by induction hypothesis, For $r > i$, $p_M(r)$ and $\hat{p}_M(r)$ differs only when $r = v_k$ with $p_M(v_k) = b_{v_k}$ and $\hat{p}_M(v_k) = b_{v_j}$. Then, by (\ref{eq:700}), (\ref{eq:702}) and $\hat{t}_i = t_i - \mathds{1}\left(b_{v_j} < p_M(i) < b_{v_k}\right)$ we have
\[
p_M(i) = \hat{t}_i + M - i - \sum_{r = i + 1}^{M}\mathds{1}\left(\hat{p}_M(r) > p_M(i)\right)
\]
Hence by Lemma \ref{LE612}, it follows that for $v_j < i < v_k$, we have $\hat{p}_M(i) = p_M(i) = b_i$. The remaining cases when $i = v_j$ and $1 \le i < v_j$ can be proved in a similar fashion. Here we omit their proofs. Therefore we have shown that $\{\hat{p}_i(i) = \hat{t}_i : i \in [M]\}$ implies $\hat{p}_M = p_M \ccirc (v_j, v_k) = \bs{b} \ccirc (v_j, v_k)$.

To prove Claim \ref{claim1}, note that conditioned on $F = \bs{v}$, the random variables $\{p_i(i) - C \cdot \mathds{1}(p_i(i) > C) \}_{i \in B}$ are independent with truncated geometric distributions. To see this, for each $i \in B$ define the events
\[
A_i \coloneqq \{p_i(i) \le C\}, \qquad \bar{A}_i \coloneqq \{p_i(i) > C\}.
\]
Note that, for $i \in \bs{v}$, we have
\begin{align}
 &\p\left(\{p_i(i) = t_i : i \in B \} \, \big| \, F = \bs{v} \right) \nonumber\\
=\,&\p\left( \{p_i(i) = t_i : i \in B\} \, \big| \, (\cap_{i \in \bs{v}} \bar{A}_i) \cap (\cap_{i \in B\setminus \bs{v}} A_i)  \right) \nonumber\\
=& \frac{\p\left((\cap_{i \in B}\{p_i(i) = t_i\})\cap (\cap_{i \in \bs{v}} \bar{A}_i) \cap (\cap_{i \in B\setminus \bs{v}} A_i) \right)}{\p\left((\cap_{i \in \bs{v}} \bar{A}_i) \cap (\cap_{i \in B\setminus \bs{v}} A_i) \right)}\nonumber\\
=\,&\frac{\prod_{i \in \bs{v}} \p\left(\{p_i(i) = t_i\} \cap \bar{A}_i\right)\cdot \prod_{i \in B\setminus \bs{v}} \p\left((\{p_i(i) = t_i\} \cap A_i \right)}{ \prod_{i \in \bs{v}} \p\left(\bar{A}_i\right)\cdot \prod_{i \in B\setminus \bs{v}} \p\left(A_i \right)} \label{eq:6751}\\
=\,& \prod_{i \in \bs{v}} \p\left(p_i(i) = t_i \, \big| \, \bar{A}_i\right)\cdot \prod_{i \in B\setminus \bs{v}} \p\left(p_i(i) = t_i \, \big| \, A_i \right)\nonumber\\
=\,& \prod_{i \in \bs{v}} \p\left(p_i(i) = t_i \, \big| \, p_i(i) > C \right) \cdot \prod_{i \in B \setminus \bs{v}} \p\left(p_i(i) = t_i \, \big| \, p_i(i) \le C \right). \label{eq:6752}
\end{align}

Hence, we have
\begin{align}
\p\left(\left\{p_i(i) = t_i : i \in [M] \right\}\, \big|\, F = \bs{v} \right) &= c\cdot q^{\sum_{i = 1}^{M} t_i - lC}, \label{eq:676}\\
\p\left(\left\{ p_i(i) = \hat{t}_i : i \in [M] \right\}\, \big|\, F = \bs{v} \right) &= c\cdot q^{\sum_{i = 1}^{M} \hat{t}_i - lC}, \label{eq:677}
\end{align}
Here $c$ is a normalizing constant. By Corollary \ref{C611}, we have
\[
 \sum_{i = 1}^{M} t_i = \frac{(M+1)M}{2} - \inv(\bs{b}), \quad \sum_{i = 1}^{M} \hat{t}_i = \frac{(M+1)M}{2} - \inv(\bs{b}\ccirc (v_j, v_k)).
\]
Since $b_{v_j} < b_{v_k}$ implies $\inv(\bs{b}) < \inv(\bs{b}\ccirc (v_j, v_k))$, we have $\sum_{i = 1}^{M} t_i > \sum_{i = 1}^{M} \hat{t}_i$. Thus, by (\ref{eq:676}) and (\ref{eq:677}),
\[
\p\left(\left\{p_i(i) = t_i : i \in [M] \right\}\, \big|\, F = \bs{v} \right) <
\p\left(\left\{ p_i(i) = \hat{t}_i : i \in [M] \right\}\, \big|\, F = \bs{v} \right).
\]
By \ref{eq:659} and (\ref{eq:661}), Claim \ref{claim1} follows. 
\end{proof}

Based on Claim \ref{claim1} and assuming the setting of Lemma \ref{LE67}, we next prove the following claim.

\begin{claim}\label{claim2}
For any $\kappa \in S_M$ and any $w \in [M - 1]$ such that $\kappa^{-1}(w) < \kappa^{-1}(w+1)$, there exists a coupling of two $q$-Mallows process $\{\bar{p}_i\}$ and $\{\hat{p}_i\}$ such that the following are satisfied.
\begin{itemize}
\item With $\bar{F} \coloneqq \left\{ i \in B : \bar{p}_i(i) > C \right\}$ and $\hat{F} \coloneqq \left\{ i \in B : \hat{p}_i(i) > C \right\}$, we have $\bar{F} = \hat{F}$.
\item $\lis((\bar{p}_M)_{\bar{F}}, \kappa_{\bar{F}}) \le \lis((\hat{p}_M)_{\hat{F}}, ((w, w+1)\ccirc \kappa)_{\hat{F}})$.
\end{itemize}
\end{claim}

\begin{proof}[Proof of Claim \ref{claim2}]
By Lemma \ref{LE611}, we know that the values of $\{p_i(i)\}_{i \in [M]}$ are determined by $p_M$. Hence, to construct a coupling of $\{\bar{p}_i\}$ and $\{\hat{p}_i\}$, it suffices to define a coupling of $(\bar{p}_M, \hat{p}_M)$.

Let $\{p_i\}$ be a $q$-Mallows process. Define $F \coloneqq \left\{ i \in B : p_i(i) > C \right\}$. Let $\bs{v} = \{v_1, \cdots, v_l\}$ be a sequence of increasing indices in $[M]$. Conditioned on $F = \bs{v}$, we define $(\bar{p}_M, \hat{p}_M)$ as follows.
\begin{itemize}
\item[] \textbf{Case 1}: If $\kappa^{-1}(w) \notin \bs{v}$ or $\kappa^{-1}(w+1) \notin \bs{v}$, define $\bar{p}_M = \hat{p}_M = p_M$.
\item[] \textbf{Case 2}: If $\kappa^{-1}(w) = v_j$ and $\kappa^{-1}(w+1) = v_k$, note that we can partition $S_M$ into pairs of permutations $\{\bs{b}, \bs{b}\ccirc (v_j, v_k)\}$ with $b_{v_j} < b_{v_k}$. Then, first choose a pair of permutations $\{\bs{b}, \bs{b}\ccirc (v_j, v_k)\}$ with probability $\p\left(p_M = \bs{b}\,\big|\, F = \bs{v}\right) + \p\left(p_M = \bs{b}\ccirc(v_j, v_k)\,\big|\, F = \bs{v}\right)$. Flip a coin with probability of head being
\begin{equation}\label{eq:678}
h \coloneqq \frac{2\cdot\p\left(p_M = \bs{b}\,\big|\, F = \bs{v}\right)}{\p\left(p_M = \bs{b}\,\big|\, F = \bs{v}\right) + \p\left(p_M = \bs{b}\ccirc(v_j, v_k)\,\big|\, F = \bs{v}\right)}.
\end{equation}
If the outcome is tail, define $\bar{p}_M = \hat{p}_M = \bs{b} \ccirc (v_j, v_k)$. If the outcome is head, then, with equal probability, define either $\bar{p}_M =\bs{b}$, $\hat{p}_M = \bs{b}\ccirc(v_j, v_k)$ or $\bar{p}_M = \bs{b}\ccirc (v_j, v_k)$, $\hat{p}_M = \bs{b}$.
\end{itemize}
For the first case, note that $\kappa^{-1}(w) \notin \bs{v}$ or $\kappa^{-1}(w+1) \notin \bs{v}$ implies  $\kappa_{\bs{v}} = ((w, w+1)\ccirc \kappa)_{\bs{v}}$. Hence, by setting $\bar{p}_M = \hat{p}_M = p_M$, the two conditions in the claim are satisfied trivially. For the second case, note that by Claim \ref{claim1}, the probability of being head $h$ defined in (\ref{eq:678}) is no greater than 1. As shown in the proof of Claim \ref{claim1}, when one of $\bar{p}_M$ and $\hat{p}_M$ equals $\bs{b}$ and the other equals $\bs{b}\ccirc(v_j, v_k)$, we have $\bar{F} = \hat{F} = \bs{v}$. Moreover, it is easy to verify that $((w, w+1)\ccirc \kappa)_{\bs{v}} = \kappa_{\bs{v}}\ccirc (j, k)$ and $(\bs{b}\ccirc (v_j, v_k))_{\bs{v}} = \bs{b}_{\bs{v}}\ccirc (j, k)$. Hence, when the outcome of the coin is head we have either $\bar{p}_M =\bs{b}$, $\hat{p}_M = \bs{b}\ccirc(v_j, v_k)$ or $\bar{p}_M = \bs{b}\ccirc (v_j, v_k)$, $\hat{p}_M = \bs{b}$. In either case, we can verify that
\begin{equation}\label{eq:6781}
\lis((\bar{p}_M)_{\bs{v}}, \kappa_{\bs{v}}) = \lis((\hat{p}_M)_{\bs{v}}, ((w, w+1)\ccirc \kappa)_{\bs{v}}).
\end{equation}
For example, if $\bar{p}_M =\bs{b}$, $\hat{p}_M = \bs{b}\ccirc(v_j, v_k)$, we have
\begin{align*}
\lis((\bar{p}_M)_{\bs{v}}, \kappa_{\bs{v}}) &= \lis(\bs{b}_{\bs{v}}, \kappa_{\bs{v}}),\\
\lis((\hat{p}_M)_{\bs{v}}, ((w, w+1)\ccirc \kappa)_{\bs{v}}) &= \lis((\bs{b}\ccirc (v_j, v_k))_{\bs{v}}, \kappa_{\bs{v}}\ccirc (j, k))\\
&= \lis(\bs{b}_{\bs{v}}\ccirc (j, k), \kappa_{\bs{v}}\ccirc (j, k))\\
& = \lis(\bs{b}_{\bs{v}}, \kappa_{\bs{v}}).
\end{align*}
For the other case, (\ref{eq:6781}) can be verified similarly.
When the outcome is tail, we need to show that
\begin{equation}\label{eq:679}
\lis((\bs{b} \ccirc (v_j, v_k))_{\bs{v}}, \kappa_{\bs{v}}) \le \lis((\bs{b} \ccirc (v_j, v_k))_{\bs{v}}, ((w, w+1)\ccirc \kappa)_{\bs{v}}).
\end{equation}
Note that we have $(\bs{b} \ccirc (v_j, v_k))_{\bs{v}} = \bs{b}_{\bs{v}} \ccirc (j, k)$ and $((w, w+1)\ccirc \kappa)_{\bs{v}} = (r, r+1)\ccirc \kappa_{\bs{v}}$, where $r$ is the rank of $w$ in $\kappa$ restricted to $\bs{v}$. Moreover, we have $(\kappa_{\bs{v}})^{-1}(r) = j < k = (\kappa_{\bs{v}})^{-1}(r+1)$. Hence by Lemma 2.3 in \cite{Ke2}, we have
\begin{align}
\lis((\bs{b} \ccirc (v_j, v_k))_{\bs{v}}, \kappa_{\bs{v}}) &= \lis( \bs{b}_{\bs{v}} \ccirc (j, k), \kappa_{\bs{v}}) \label{eq:680} \\ 
&= \lis(\bs{b}_{\bs{v}} \ccirc (j, k) \ccirc (\kappa_{\bs{v}})^{-1}, id) \nonumber \\
\lis((\bs{b} \ccirc (v_j, v_k))_{\bs{v}}, ((w, w+1)\ccirc \kappa)_{\bs{v}}) &= \lis( \bs{b}_{\bs{v}} \ccirc (j, k), (r, r+1)\ccirc \kappa_{\bs{v}}) \label{eq:681}\\
&= \lis(\bs{b}_{\bs{v}} \ccirc (j, k) \ccirc (\kappa_{\bs{v}})^{-1}, (r, r+1)). \nonumber
\end{align}
Here $id$ denotes the identity in $S_l$. Note that
\begin{align}
 &\bs{b}_{\bs{v}} \ccirc (j, k) \ccirc (\kappa_{\bs{v}})^{-1}(r) = \bs{b}_{\bs{v}} \ccirc (j, k)(j) = \bs{b}_{\bs{v}}(k), \\
 &\bs{b}_{\bs{v}} \ccirc (j, k) \ccirc (\kappa_{\bs{v}})^{-1}(r+1) = \bs{b}_{\bs{v}} \ccirc (j, k)(k) = \bs{b}_{\bs{v}}(j).
\end{align}
Since $b_{v_j} < b_{v_k}$, we have $\bs{b}_{\bs{v}}(j) < \bs{b}_{\bs{v}}(k)$, which means $\{r, r+1\}$ form an inversion for the permutation $\bs{b}_{\bs{v}} \ccirc (j, k) \ccirc (\kappa_{\bs{v}})^{-1}$. Hence (\ref{eq:679}) follows from (\ref{eq:680}) and (\ref{eq:681}).

Finally, it can be easily verified that $\bar{p}_M$ and $\hat{p}_M$ thus defined have the right marginal distribution, i.e.\,both $\bar{p}_M$ and $\hat{p}_M$ have the same distribution as $p_M$.

\end{proof}

Before we complete the proof of Lemma \ref{LE67}, we introduce the following partial order on $S_n$.

\begin{definition}\label{DE2}
The left weak Bruhat order $(S_n, \le_L)$ is defined as the transitive closure of the relations
\[
\pi \le_L \tau \quad \text{if} \quad \tau = (i, i+1) \ccirc \pi \ \text{ and }\  \inv(\tau) = \inv(\pi) + 1.
\]
\end{definition}

\begin{proof}[Proof of Lemma \ref{LE67}]
Let $id_M^r$ denote the reversal of identity in $S_M$. Considering the poset $(S_M, \le_L)$, it follows from Definition \ref{DE2} that $id_M^r$ is the maximum element in $(S_M, \le_L)$. Hence for any permutation $\kappa \neq id_M^r$, we can find a sequence of permutations $\{\kappa_i\}$ such that
\[
\kappa  = \kappa_0 \le_L \kappa_1 \le_L \cdots \le_L \kappa_m = id_M^r,
\]
and $\kappa_{i+1}$ covers $\kappa_{i}$, i.e.\,there exists $w \in [M-1]$ such that $(w, w+1) \ccirc \kappa_i = \kappa_{i+1}$ and $\inv(\kappa_{i+1}) = \inv(\kappa_{i}) + 1$. Note that here $m = \frac{M(M-1)}{2} - \inv(\kappa)$. Then by Claim~\ref{claim2} and induction on $m$, it can be shown that there exists a coupling, denoted by $\mathcal{C}_{\kappa}$, of two $q$-Mallows processes $\{\bar{p}_i\}$ and $\{\hat{p}_i\}$ such that the following are satisfied.
\begin{itemize}
\item With $\bar{F} \coloneqq \left\{ i \in B : \bar{p}_i(i) > C \right\}$ and $\hat{F} \coloneqq \left\{ i \in B : \hat{p}_i(i) > C \right\}$, we have $\bar{F} = \hat{F}$.
\item $\lis((\bar{p}_M)_{\bar{F}}, \kappa_{\bar{F}}) \le \lis((\hat{p}_M)_{\hat{F}}, (id_M^r)_{\hat{F}})$.
\end{itemize}
Note that, by Definition \ref{D31}, for any increasing sequence of indices $F$, we have
\begin{align}
\lis(\pi_F, \tau_F) &= \lis(\pi(F), \tau(F)) = \lis(\bar{p}_n(F), p'_n(F))\label{eq:682}\\
 &= \lis(\bar{p}_M(F), p'_M(F)) = \lis((\bar{p}_M)_F, (p'_M)_F), \nonumber\\
\lis(\hat{\pi}_F) &= \lis(\hat{\pi}(F), id_n(F)) = \lis(\hat{p}_n(F), (id_n^r)(F))\label{eq:683}\\
&= \lis(\hat{p}_M(F), (id_M^r)(F)) = \lis((\hat{p}_M)_F, (id_M^r)_F). \nonumber
\end{align}
Here $id_n$ denotes the identity in $S_n$. Hence by (\ref{eq:682}) and (\ref{eq:683}) we have
\begin{equation}\label{eq:684}
\lis((\bar{p}_M)_F, (p'_M)_F) \le \lis((\hat{p}_M)_F, (id_M^r)_F) \Rightarrow \lis(\pi_F, \tau_F) \le \lis(\hat{\pi}_F).
\end{equation}
We define the coupling $\{\bar{p}_i\}$, $\{p'_i\}$ and $\{\hat{p}_i\}$ as follows. For any $i > M$, we simply let $\bar{p}_i$, $p'_i$ and $\hat{p}_i$ be i.i.d.\,truncated geometric distributed. For $1 \le i \le M$, let $p'_M \sim \mu_{M, q}$. Conditioned on $p'_M = \kappa$, define $\{\bar{p}_i\}$ and $\{\hat{p}_i\}$ such that they have joint distribution $\mathcal{C}_{\kappa}$. The lemma follows from (\ref{eq:684}) and the property of $\mathcal{C}_{\kappa}$.
\end{proof}

\section{Central Limit Theorem for LCS}
In this section, we prove a central limit theorem for the LCS of two independent Mallows permutations when the parameters $0 < q, q' < 1$ are fixed. The proof of Theorem \ref{M5} is based on the approach developed in \cite{Basu} in which the authors prove a central limit theorem for the LIS of a Mallows permutation. The idea is to construct a regenerative process such that we can bound the LCS by the sum of i.i.d.\,random variables defined in terms of the process.

\subsection{Constructing Mallows Permutations}
For a given parameter $0 < q < 1$, Gnedin and Olshanski \cite{Gnedin} constructed an infinite Mallows permutation with parameter $q$ on $\mathbb{N}$ by an insertion process, which we will refer to as Mallows($q$) process. This gives us another method for generating finite sized Mallows permutations. Given an i.i.d.\,sequence $\{Z_i\}_{i \ge 1}$ of Geom($1-q$) variables, construct a permutation $\tilde{\Pi}$ of the natural numbers inductively according the following rule: Set $\tilde{\Pi}(1) = Z_1$. For $i > 1$, set $\tilde{\Pi}(i) = k$ where $k$ is the $Z_i$-th number in the increasing order from the set $\mathbb{N}\setminus\{\tilde{\Pi}(j) : 1 \le j < i\}$. For example, suppose that the realizations of the first five independent geometric random variables are $Z_1 = 4, Z_2 = 4, Z_3 = 1, Z_4 = 2, Z_5 = 3$. Then we have $\tilde{\Pi}(1) = 4$, $\tilde{\Pi}(2) = 5$, $\tilde{\Pi}(3) = 1$, $\tilde{\Pi}(4) = 3$ and $\tilde{\Pi}(5) = 7$. We represent the process step-by-step below. 
\[
\begin{matrix}
\underline{\phantom{000}}&\underline{\phantom{000}}&\underline{\phantom{000}}&1&\underline{\phantom{000}}&\underline{\phantom{000}}&\underline{\phantom{000}}&\underline{\phantom{000}}&\cdots\\
\underline{\phantom{000}}&\underline{\phantom{000}}&\underline{\phantom{000}}&1&2&\underline{\phantom{000}}&\underline{\phantom{000}}&\underline{\phantom{000}}&\cdots\\
3&\underline{\phantom{000}}&\underline{\phantom{000}}&1&2&\underline{\phantom{000}}&\underline{\phantom{000}}&\underline{\phantom{000}}&\cdots\\
3&\underline{\phantom{000}}&4&1&2&\underline{\phantom{000}}&\underline{\phantom{000}}&\underline{\phantom{000}}&\cdots\\
3&\underline{\phantom{000}}&4&1&2&\underline{\phantom{000}}&5&\underline{\phantom{000}}&\cdots
\end{matrix}
\]
Let $\Pi_n$ be the permutation on $[n]$ induced by $\tilde{\Pi}$, i.e.\,,$\Pi_n(i) = j$ if $\tilde{\Pi}(i)$ has rank $j$ when the set $\{\tilde{\Pi}(k) : k \in [n]\}$ is arranged in increasing order. Consider the example above when $n = 5$. Then we have $\Pi_5(1) = 3$, $\Pi_5(2) = 4$, $\Pi_5(3) = 1$, $\Pi_5(4) = 2$ and $\Pi_5(5) = 5$. The following lemma (cf. Lemma 2.1 in \cite{Basu}) says that $\Pi_n$ thus defined is Mallows distributed with parameter $q$.
\begin{lemma}\label{LE41}
Let $\tilde{\Pi}$ be an infinite Mallows($q$) permutation and let $\Pi_n$ be the induced permutation on $[n]$ as defined above. Then $\Pi_n$ is a Mallows($q$) permutation on $[n]$.
\end{lemma}

\subsection{The Regenerative Process Representation}\label{regen}
A stochastic process $\{X(t) : t \ge 0\}$ is said to be a \emph{regenerative process} if there exist regeneration times $0 \le T_0 < T_1 < T_2 < \cdots$ such that for each $k \ge 1$, the process $\{X(T_k + t) : t \ge 0\}$ has the same distribution as $\{X(T_0 + t) : t \ge 0\}$ and is independent of $\{X(t) : 0 \le t < T_k\}$. In the following, we will define a regenerative process using two independent copies of the Mallows($q$) process.

Let $\tilde{\Pi}$ and $\tilde{\Pi}'$ be two independent infinite Mallows permutations with parameters $q, q'$ respectively. Suppose for a given $m \in \mathbb{N}$ we have $\tilde{\Pi}([m]) = \tilde{\Pi}'([m]) = [m]$, i.e.\,the permutations $\tilde{\Pi}$ and $\tilde{\Pi}'$ restricted to $[m]$ define two bijections from $[m]$ to $[m]$. Define two infinite permutations $\tilde{\Pi}_{m}$ and $\tilde{\Pi}'_{m}$ as follows,
\[
\tilde{\Pi}_{m}(i) \coloneqq \tilde{\Pi}(i+m) - m, \qquad \tilde{\Pi}'_{m}\coloneqq \tilde{\Pi}'(i+m) - m, \quad \forall i \in \mathbb{N}.
\]
From the construction of $\tilde{\Pi}$ and $\tilde{\Pi}'$, it is obvious that $\tilde{\Pi}_{m}$ and $\tilde{\Pi}'_{m}$ are also infinite Mallows permutations with parameters $q$ and $q'$ respectively. Together with the independence of the geometric variables $\{Z_i\}$ as well as $\{Z'_i\}$, it follows that $\big\{\big(\tilde{\Pi}(i)-i, \tilde{\Pi}'(i) - i\big)\big\}_{i\in\mathbb{N}}$ is a regenerative process with regeneration times $0 = T_0 < T_1 < T_2 < \cdots$ where for $i > 1$ we have,
\[
T_i \coloneqq \min{\left\{j > T_{i-1} : \big\{\tilde{\Pi}(k) : k \in [j]\big\}=  \big\{\tilde{\Pi}'(k) : k \in [j]\big\} =  [j]\right\}}.
\]
Let $X_j \coloneqq T_j - T_{j-1}$ for $j \ge 1$. Clearly, $X_j$ are independent and identically distributed. For $j \ge 1$, define
\[
\Sigma_j(i) \coloneqq \tilde{\Pi}(i + T_{j-1}) - T_{j-1}, \quad \Sigma'_j(i) \coloneqq \tilde{\Pi}'(i + T_{j-1}) - T_{j-1}, \quad \forall i \in [X_j].
\]
Then, both $\Sigma_j$ and $\Sigma'_j$ are permutations of $[X_j]$. Furthermore, the $\{\Sigma_j\}_{j \in \mathbb{N}}$ are i.i.d.\,and $\{\Sigma'_j\}_{j \in \mathbb{N}}$ are i.i.d.. Let $Y_j \coloneqq \lcs(\Sigma_j, \Sigma'_j)$ i.e.\,$Y_j$ denotes the length of the longest common subsequence between $\Sigma_j$ and $\Sigma'_j$. Clearly, $\{Y_j\}_{j \in \mathbb{N}}$ are i.i.d..  Then we have the following bounds for the $\lcs$ of two independent Mallows permutation.
\begin{lemma}\label{LE42}
Let $S_n \coloneqq \min\{j: T_j \ge n\}$. Then we have
\[
\sum_{j = 1}^{S_n - 1} Y_j < \lcs(\Pi_n, \Pi'_n) \le \sum_{j = 1}^{S_n} Y_j.
\]
\end{lemma}
\begin{proof}
Given $j > 0$, let $\lcs_{[T_{j-1}+1, T_j]}(\Pi_n, \Pi'_n)$ denote the length of the longest common subsequence of $\Pi_n, \Pi'_n$ restricted on $[T_{j-1}+1, T_j]$. From the definition of $T_j$, we have $\Pi_n([T_{j-1}+1, T_j]) = \Pi'_n([T_{j-1}+1, T_j]) = [T_{j-1}+1, T_j]$. Thus, we get
\[
\sum_{j = 1}^{S_n - 1} \lcs_{[T_{j-1}+1, T_j]}(\Pi_n, \Pi'_n) < \lcs(\Pi_n, \Pi'_n) \le \sum_{j = 1}^{S_n} \lcs_{[T_{j-1}+1, T_j]}(\Pi_n, \Pi'_n).
\]
It follows from the definition of $\Sigma_j$ and $\Sigma'_j$ that there exists a bijection between the common subsequences of $\Pi_n$, $\Pi'_n$ restricted on $[T_{j-1}+1, T_j]$ and the common subsequences of $\Sigma_j$, $\Sigma'_j$. Hence we have $\lcs_{[T_{j-1}+1, T_j]}(\Pi_n, \Pi'_n) = Y_j$. The lemma follows.
\end{proof}

\subsection{Renewal Time Estimate and Proof of the CLT for LCS}
In this section, we first prove that the inter-renewal times $X_i$ as defined in the previous section have finite first and second moments, which are the conditions required to apply results from the theory of regenerative processes to show Theorem \ref{M5}. Again we follow the approach developed in \cite{Basu}, in which the authors introduce the following Markov chain.

Let $\{M_n\}_{n \ge 0}$ denote the Markov chain with the state space $\Omega = \mathbb{N}\cup\{0\}$ and the one step transition defined as follows: $M_n \coloneqq  \max\{M_{n-1}, Z_n\} - 1$ where $\{Z_i\}$ is a sequence of i.i.d.\,Geom($1-q$) variables. Likewise, for the parameter $q'$, we define a Markov chain $\{M'_n\}_{n \ge 0}$ in the same fashion, i.e., the one step transition rule is defined by $M'_n \coloneqq \max\{M'_{n-1}, Z'_n\} - 1$ where $\{Z'_i\}$ is a sequence of i.i.d.\,Geom($1-q'$) variables. Let $\{M^{\otimes}_n\}_{n \ge 0}$ denote the product chain of $\{M_n\}$ and $\{M'_n\}$. Let $R^+_0$ denote the first return time to $(0, 0)$ of this chain, i.e.
\[
R^+_0 \coloneqq \min\{k > 0: M^{\otimes}_k = (0, 0)\}.
\]
\begin{lemma}\label{LE51}
For the Markov chain $\{M^{\otimes}_n\}$ started at $M^{\otimes}_0 = (0, 0)$, the first return time $R^+_0 \overset{d}{=} T_1$. In other words, $X_i$ has the same distribution as $R^+_0$.
\end{lemma}
\begin{proof}
We couple the Markov chain $M^{\otimes}_n = (M_n, M'_n)$ with the infinite Mallows permutations $\tilde{\Pi}$, $\tilde{\Pi'}$ with parameters $q$ and $q'$ respectively by using the same i.i.d.\,sequences $\{Z_i\}$ and $\{Z'_i\}$ with $Z_i \sim\ $Geom($1-q$) and $Z'_i \sim\ $Geom($1-q'$). Under this coupling, it is easy to verify that
\[
M_n = \max_{1 \le j \le n}\big\{\tilde{\Pi}(j)\big\} - n, \qquad M'_n = \max_{1 \le j \le n}\big\{\tilde{\Pi'}(j)\big\} - n.
\]
The lemma follows from the definition of $T_1$ and $R^+_0$.
\end{proof}

We analyze the Markov chain $M^{\otimes}_n$ and the first return time $R^+_0$ in the next few lemmas.

\begin{lemma}\label{LE52}
The Markov chain $M^{\otimes}_n$ is a positive recurrent Markov chain with unique stationary distribution $\nu = (\nu_{i,j})_{i, j \ge 0}$ where
\[
\nu_{i,j} \coloneqq \frac{q^i}{\mathcal{Z}(q)\prod_{k = 1}^{i}\big(1-q^k\big)}\cdot
\frac{(q')^j}{\mathcal{Z}(q')\prod_{k = 1}^{j}\big(1-(q')^k\big)}.
\]
Here $\mathcal{Z}(q) \coloneqq 1/\prod_{k = 1}^{\infty}\big(1-q^k\big)$.
\end{lemma}
Note that $\mathcal{Z}(q)$ is finite since $\lim_{k\to\infty}\log{\big(\frac{1}{1-q^k}\big)}/q^k = 1$.
\begin{proof}
The claim follows directly from Lemma 4.2 in \cite{Basu} and the fact that $M^{\otimes}_n$ is the product chain of $M_n$ and $M'_n$.
\end{proof}

Let $R_t$ denote the first time the chain $M^{\otimes}_n$ to reach a state $(i, j)$ such that $i+j \le t$. In the following, we shall denote by $\E_{i,j}$ the expectation with respect to the chain started at the state $(i, j)$ and $\E_{\nu}$ denote the expectation with respect to the chain started from the stationary distribution.
\begin{lemma}\label{LE53}
For any $i, j \ge 0$ with $i+j > 0$, we have
\[
\E_{i, j}R_{i + j -1} \ge \E_{i, j+1}R_{i + j}, \qquad \E_{i, j}R_{i+j-1} \ge \E_{i + 1, j}R_{i + j}.
\]
\end{lemma}
\begin{proof}
By symmetry of $M_n$ and $M'_n$, it suffices to show the first inequality. We couple two chains $(M_n, M'_n)$ and $(\tilde{M}_n, \tilde{M}'_n)$ which start from $(i, j)$ and $(i, j+1)$ respectively by using the same sequences $\{Z_i\}$ and $\{Z'_i\}$. It is easily seen from the one step transition rule that, at any time $n$, we have $M_n = \tilde{M}_n$ and $0 \le \tilde{M}'_n - M'_n \le 1$. Thus we have
\[
0 \le (\tilde{M}_n + \tilde{M}'_n) - (M_n + M'_n) \le 1, \quad \forall n \ge 0.
\]
Therefore, $M_n + M'_n \le i+j-1$ implies $\tilde{M}_n + \tilde{M}'_n \le i + j$.
\end{proof}

An immediate corollary of Lemma \ref{LE53} is the following.
\begin{corollary}\label{LE54}
For any $i, j \ge 0$ with $i+j > 0$, 
\[
\max\{\E_{0, 1}R_0,\ \E_{1, 0}R_0\} \ge \E_{i, j}R_{i+j-1}.
\]
\end{corollary}
The positive recurrence of the chain $M^{\otimes}_n$ implies that $\E_{0, 1}R_0$ and $\E_{1, 0}R_0$ are finite. Let $\eta \coloneqq \max\{\E_{0, 1}R_0,\ \E_{1, 0}R_0\}$. 

\begin{lemma}\label{LE55}
For any $i, j \ge 0$ with $i+j > 0$, we have
\[
\E_{i, j}R_{0} \le (i+j) \eta.
\]
\end{lemma}
\begin{proof}
We proof this lemma by induction on the sum of $i$ and $j$. When $i + j = 1$, the claim holds trivially. Suppose the claim holds for any $\{i, j \ge 0 : i + j \le k\}$. Given $s, t$ with $s + t = k+1$, by the Markov property, we have
\begin{align*}
\E_{s, t}R_{0} &= \sum_{n \ge 1}\sum_{i + j \le k}
                  \left(n + \E_{i, j}R_{0}\right)\cdot \p_{s, t}\left(R_{k} = n, M^{\otimes}_n = (i, j)\right)\\
               &\le \sum_{n \ge 1}\sum_{i + j \le k}
                  \left(n + k\eta\right)\cdot \p_{s, t}\left(R_{k} = n, M^{\otimes}_n = (i, j)\right)\\
               &= k\eta +\sum_{n \ge 1}\sum_{i + j \le k}
                                 n\cdot \p_{s, t}\left(R_{k} = n, M^{\otimes}_n = (i, j)\right)\\
               &= k\eta +\sum_{n \ge 1} n\cdot \p_{s, t}\left(R_{k} = n\right)\\
               &= k\eta + \E_{s,t}R_k\\
               &\le (k + 1)\eta
\end{align*}
Here the first inequality follows from induction hypothesis and the last inequality follows from Corollary \ref{LE54}.
\end{proof}

\begin{lemma}\label{LE56}
For the Markov chain $M^{\otimes}_n$, $\E_{\nu}R_0 < \infty$.
\end{lemma}
\begin{proof}
By Lemma 4.2 in \cite{Basu}, the stationary distributions of $M_n$ and $M'_n$ are
\begin{align}
\mu_{i} &\coloneqq \sum_{j = 0}^{\infty}\nu_{i, j} = \frac{q^i}{\mathcal{Z}(q)\prod_{k = 1}^{i}\big(1-q^k\big)},\\
\mu'_{j} &\coloneqq \sum_{i=0}^{\infty} \nu_{i, j} = \frac{(q')^j}{\mathcal{Z}(q')\prod_{k = 1}^{j}\big(1-(q')^k\big)}.
\end{align}

Note that we have
\begin{align}
\E_{\nu}R_0 = \sum_{i, j \ge 0}\nu_{i, j}\E_{i, j}R_0 
            &\le \sum_{i, j \ge 0}\nu_{i, j} (i + j)\eta \label{eq:le56}\\
            &= \eta \sum_{i = 0}^{\infty}i\mu_i + \eta\sum_{j = 0}^{\infty} j\mu'_j. \nonumber
\end{align}
%%%%%%%%%%%%&= \eta \sum_{i, j \ge 0}i\nu_{i, j} + \eta\sum_{i, j \ge 0}j\nu_{i, j}\\
By the definition of $\mathcal{Z}(q)$, we have $\mu_i < q^i/\mathcal{Z}(q)^2$. Hence $\sum_{i = 0}^{\infty} i\mu_i < \infty$. Similarly we also have $\sum_{j = 0}^{\infty} j\mu'_j < \infty$. Therefore, by (\ref{eq:le56}), $\E_{\nu}R_0 < \infty$.
\end{proof}

In the next lemma, we show that the first and second moments of the first return time $R^+_0$ are finite by using Kac's formula.
\begin{lemma}\label{LE57}
\[
\E_{0, 0} R^+_0 < \infty,\qquad \E_{0, 0} (R^+_0)^2 < \infty
\]
\end{lemma}
\begin{proof}
It is a basic fact about Markov chains that $\E_{0, 0} R^+_0 = \frac{1}{\nu_{0, 0}}$. By Lemma \ref{LE52} and the finiteness of $\mathcal{Z}(q)$ and $\mathcal{Z}(q')$, we have $\frac{1}{\nu_{0, 0}} = \mathcal{Z}(q)\cdot\mathcal{Z}(q') < \infty$. The finiteness of the second moment of $R^+_0$ follows from Lemma \ref{LE56} and the following consequence of Kac's formula (cf. (2.21) in \cite{Aldous}),
\[
\E_{0, 0}(R^+_0)^2 = \frac{2E_{\nu}(R_0) + 1}{\nu_{0, 0}}.
\]
\end{proof}

In the remainder of this section, we complete the proof of Theorem \ref{M5} by using the following version of central limit theorem due to Anscombe.
\begin{theorem}[Anscombe's Theorem]\label{anscombe}
Let $\{X_i\}_{i \ge 1}$ be a sequence of i.i.d. random variables with mean 0 and positive, finite variance $\sigma^2$. For $n \ge 1$, let $Q_n \coloneqq \sum_{i = 1}^{n} X_i$. Suppose $\{N(t), t \ge 0\}$ is a family of positive integer-valued random variables such that for some $0 < c < \infty$, 
\[
\frac{N(t)}{t}\overset{p}{\longrightarrow} c \quad \text{ as } t \to \infty.
\]
Then,
\[
\frac{Q_{N(t)}}{\sqrt{t}}\overset{d}{\longrightarrow} \mathcal{N}(0, c\sigma^2) \quad \text{ as } t \to \infty.
\]
\end{theorem}
Recall that in section \ref{regen}, we define $X_i$ to be the inter-renewal times and $S_n = \min\{j: \sum_{i = 1}^j X_i \ge n\}$.
\begin{lemma}\label{LE58}
For $\nu_{0, 0}$ as defined in Lemma \ref{LE52}, 
\[
\frac{S_n}{n} \overset{a.s.}{\longrightarrow} \nu_{0, 0}.
\]
\end{lemma}
\begin{proof}
Observer that
\[
\frac{\sum_{j = 1}^{S_n - 1} X_j}{S_n} \le \frac{n}{S_n} \le \frac{\sum_{j = 1}^{S_n} X_j}{S_n}.
\]
As $n\to\infty$, by the strong law of large numbers, both the left and right hand sides of the above inequality converge almost surely to $\nu_{0, 0}^{-1}$.
\end{proof}

As our last step in preparation for the proof of Theorem \ref{M5}, we introduce the following basic result (cf. Lemma 5.5 in \cite{Basu}).
\begin{lemma}\label{LE59}
Let $W_1, W_2, \ldots$ be an i.i.d.\,sequence of non-negative random variables with $\E W_i^2 < \infty$. Then we have for any constant $c > 0$,
\[
\frac{\max_{1 \le i \le cn}W_i}{\sqrt{n}} \overset{p}{\longrightarrow} 0.
\]
\end{lemma}

We assume the notations defined in section \ref{regen}.  Let $a \coloneqq \nu_{0, 0} \E(Y_1)$ and $\delta^2 \coloneqq \var(Y_1 - a X_1)$. Since $1 \le Y_1 \le X_1$, we have $|Y_1 - a X_1| < (1 + a)X_1$. Hence by Lemma \ref{LE51} and Lemma \ref{LE57}, we have $\delta^2 < \infty$. Trivially, $\delta^2 > 0$ since $Y_1$ is clearly not constant. Hence, using Theorem \ref{anscombe} and Lemma \ref{LE58}, we can show the following regenerative version of central limit theorem.

\begin{theorem}[Regenerative CLT]\label{CLT}
Let $(X_i, Y_i)_{i \ge 1}$ and $S_n$ be as defined in section \ref{regen}. Let $Q_{S_n} \coloneqq \sum_{i = 1}^{S_n} Y_i$. Then we have
\[
\frac{Q_{S_n} - an}{\sqrt{n}} \overset{d}{\longrightarrow} \mathcal{N}\left(0, \delta^2\nu_{0,0}\right).
\]
\end{theorem}

\begin{proof}
Define $\tilde{Q}_{S_n} = \sum_{i = 1}^{S_n} (Y_i - a X_i)$. Then, by Theorem \ref{anscombe} we have
\begin{equation}\label{eq:clt1}
\frac{\tilde{Q}_{S_n}}{\sqrt{n}} \overset{d}{\longrightarrow} \mathcal{N}\left(0, \delta^2\nu_{0,0}\right).
\end{equation}
By the definition of $S_n$, we have
\begin{equation}\label{eq:clt2}
\tilde{Q}_{S_n} \le Q_{S_n} - an \le \tilde{Q}_{S_n} + a \cdot X_{S_n} \le \tilde{Q}_{S_n} + a\cdot \max_{1 \le i \le n}X_i.
\end{equation}
Here the last inequality follows since $S_n \le n$. By Lemma \ref{LE59}, we have
\[
\frac{\max_{1 \le i \le n}X_i}{\sqrt{n}} \overset{p}{\longrightarrow} 0.
\]
The theorem follows from (\ref{eq:clt1}) and (\ref{eq:clt2}).
\end{proof}

\begin{proof}[Proof of Theorem \ref{M5}]
It follows from Lemma \ref{LE42} that
\[
\frac{Q_{S_n} - an}{\sqrt{n}} - \frac{Y_{S_n}}{\sqrt{n}}\le \frac{\lcs(\Pi_n, \Pi'_n) - an}{\sqrt{n}} \le \frac{Q_{S_n} - an}{\sqrt{n}}.
\]
Since $1 \le Y_i \le X_i$, we have $\E (Y_i^2) < \E (X_i^2) < \infty$ by Lemma \ref{LE57}. Hence, by Lemma \ref{LE59}, it follows that
\[
\frac{\max_{1 \le i \le n}Y_i}{\sqrt{n}} \overset{p}{\longrightarrow} 0.
\]
Since $S_n \le n$, we have $Y_{S_n} \le \max_{1 \le i \le n}Y_i$. Thus
\[
\frac{Y_{S_n}}{\sqrt{n}} \overset{p}{\longrightarrow} 0.
\]
Therefore, by setting $\sigma \coloneqq \delta\sqrt{\nu_{0,0}}$, it follows from Theorem \ref{CLT} that
\begin{equation}\label{equ:last}
\frac{\lcs(\Pi_n, \Pi'_n) - an}{\sigma\sqrt{n}} \overset{d}{\longrightarrow} \mathcal{N}(0, 1).
\end{equation}
Theorem \ref{M5} follows from (\ref{equ:last}) and Lemma \ref{LE41}.
\end{proof}

\section*{Acknowledgements} The authors were supported in part by NSF grant DMS-1261010, an NSF CAREER Grant DMS-1554783 and a Sloan Research Fellowship.

\bibliographystyle{amsplain}
\bibliography{sampart}

\end{document}